\documentclass{article}
\usepackage{graphicx} 
\usepackage{amsmath}
\usepackage{amssymb}
\usepackage{amsthm}
\usepackage{mathrsfs}
\usepackage{CJKutf8}
\usepackage[utf8]{inputenc}
\usepackage{diagbox}
\usepackage[figuresright]{rotating} 
\usepackage[section]{placeins}
\usepackage{makecell}
\usepackage{longtable}
\allowdisplaybreaks[4]
\usepackage{xcolor}
\usepackage{embrac}
\usepackage{xpatch}
\usepackage{ytableau}
\usepackage{tikz}
\usetikzlibrary{positioning}

\newcommand{\sgn}{\operatorname{sgn}}

\newtheorem{theorem}{Theorem}[section]

\newtheorem{lemma}[theorem]{Lemma}

\newtheorem{corollary}[theorem]{Corollary}

\def\de{\begin{equation}}
	\def\ee{\end{equation}}

\title{On large genus asymptotics of certain Hurwitz numbers}
\author{Xiang Li\footnote{Email: lxiang1993@ustc.edu.cn.}}
\begin{document}
	\renewcommand{\today}{}
	\maketitle
	\begin{CJK}{UTF8}{gbsn}
		\begin{abstract}
			In this paper, based on the value of central character on the transposition, we find structure and large genus asymptotics of certain Hurwitz numbers.
		\end{abstract}
		\section{Introduction}
		The notion of Hurwitz numbers was introduced in \cite{Hur1}, \cite{Hur2}. The question is to count the weighted number $H^*_d(\theta^{(1)}, \dots , \theta^{(n)})$ of ramified coverings of degree $d$ of the Riemann sphere $\mathbb{P}^1$ with the ramification profiles $\theta^{(1)},\dots,\theta^{(n)}\vdash d$. Here, $\theta\vdash d$ denotes a partition $\theta$ of weight $d$. The ramified covering is called connected if the upper Riemann surface is connected. The weighted number of connected ramified coverings of genus $g$ and degree $d$ with the ramification profiles $\theta^{(1)},\dots,\theta^{(n)}\vdash d$ is called connected Hurwitz numbers (for short Hurwitz numbers), denoted by $H_{g,d}(\theta^{(1)}, \dots , \theta^{(n)})$. 
		
		For a partition $\theta=(\theta_1,\dots,\theta_{l})\vdash d$ with $\theta_1\geq\theta_2\geq\cdots\geq\theta_{l}> 0$, denote $|\theta|=\sum_{i=1}^{l}\theta_{i}=d,l(\theta)=l$, $l^*(\theta)=d-l$ and $z_\theta=\prod\limits_i m_i(\theta)! i^{m_i}$ with $m_i(\theta)$ being the multiplicity of $i$ in $\theta$.
		
		In this paper, we will prove
		\begin{theorem}\label{cH2}
			For any fixed $d\geq5,s\geq0$ and $\mu^{(1)},\dots,\mu^{(s)}\vdash d$, we have
			\begin{align}
				&H_{g,d}(\mu^{(1)},\dots,\mu^{(s)},2\,1^{d-2},2\,1^{d-2},\dots)\nonumber\\ =&\frac{2}{d!^2}\prod_{i=1}^s\frac{d!}{z_{\mu^{(i)}}} \sum_{1\leq m\leq \tbinom{d}{2}}b(\mu^{(1)},\dots,\mu^{(s)},m) m^{2g+2d-\sum_{i=1}^s l^*(\mu^{(i)})-2}, \label{cHv}
			\end{align}
			where $b(\mu^{(1)},\dots,\mu^{(s)},m)$ are rational numbers satisfying
			
			1. $b(\mu^{(1)},\dots,\mu^{(s)},\tbinom{d}{2})=1$; 
			
			2. $b(\mu^{(1)},\dots,\mu^{(s)},m)=0$, for $\tbinom{d-1}{2}<m<\tbinom{d}{2}$;
			
			3. $b(\mu^{(1)},\dots,\mu^{(s)},\tbinom{d-1}{2})=-d^{2-s}\prod_{i=1}^{s}m_1(\mu^{(i)})$;
			
			4. $b(\mu^{(1)},\dots,\mu^{(s)},\frac{d(d-3)}{2})=(d-1)^{2-s}\prod_{i=1}^{s}(m_1(\mu^{(i)})-1)$.
		\end{theorem}	
		For the case $s=0$, Theorem \ref{cH2} was found and proved by Hurwitz \cite{Hur1} using representation of the symmetric group. This proof was briefly reviewed in \cite{DYZ}. 
		Our proof of Theorem~\ref{cH2} for general $s$ is along this line. For the case $s=1$, Do-He-Robertson \cite{DHR} proved the statement 1 of Theorem \ref{cH2}, and conjectured the statement 2 which was later proved by Yang \cite{Y}. 
		For the case $s=2$, Do-He-Robertson \cite{DHR} also deduced the statement 1 with $\forall\,\mu^{(1)}\vdash d ,\mu^{(2)}=(k^\frac{d}{k})\vdash d$ for $k>0$. 
		In \cite{L3}, we generalize the above theorem to $H_{g,d}(\mu^{(1)},\dots,\mu^{(s)},r\,1^{d-r},r\,1^{d-r},\dots)$ and to the situation when the genus of the underlying Riemann surface could be bigger than zero. Some other general results are also obtained in \cite{ALR}.
		
		For the case $s=0$, Dubrovin-Yang-Zagier \cite{DYZ} also gave new recursions for the Hurwitz numbers by simplifying the so-called Pandharipande equation, and in this way gave a new proof of the statements 1, 2 of 
		the above theorem. For the cases $s=1,2$, by using the method of \cite{DYZ}, we also achieve in \cite{L1} another proof of part of the above theorem.

		The following corollary easily follows from Theorem~\ref{cH2}.
		\begin{corollary}\label{cH3}
			For any fixed $d\geq5,s\geq0$ and $\mu^{(1)},\dots,\mu^{(s)}\vdash d$, the asymptotics of $ H_{g,d}(\mu^{(1)},\dots,\mu^{(s)},2\,1^{d-2},2\,1^{d-2},\dots)$ is given by
			\begin{align}
				H_{g,d}(\mu^{(1)},\dots,&\mu^{(s)},2\,1^{d-2},2\,1^{d-2},\dots) \sim\frac{2\cdot d!^{s-2}}{z_{\mu^{(1)}} \cdots z_{\mu^{(s)}}}\Big(\binom{d}{2}^{2g+2d-\sum_{i=1}^s l^*(\mu^{(i)})-2}\nonumber\\
				&-d^{2-s}\prod_{i=1}^{s}m_1(\mu^{(i)})\binom{d-1}{2}^{2g+2d-\sum_{i=1}^s l^*(\mu^{(i)})-2}\nonumber\\
				&+(d-1)^{2-s}\prod_{i=1}^{s}(m_1(\mu^{(i)})-1)(\frac{d(d-3)}{2})^{2g+2d-\sum_{i=1}^s l^*(\mu^{(i)})-2}\nonumber\\
				&+o((\frac{d(d-3)}{2})^{2g+2d-\sum_{i=1}^s l^*(\mu^{(i)})-2})   \Big),\,\,g\rightarrow \infty. 
			\end{align}
		\end{corollary}

		\section{The proof of Theorem~\ref{cH2}}
		In this section, we prove Theorem~\ref{cH2}.
		\begin{proof}[\bf{Proof of Theorem~\ref{cH2}}]
			It is known that $H_d^{*}(\theta^{(1)}, \dots , \theta^{(n)})$ has the following formula \cite{Burn,Fro}:
			\begin{equation}\label{Hurwitz1}
				H_d^{*}(\theta^{(1)}, \dots , \theta^{(n)}) = \sum_{ \lambda\vdash d}(\frac{\dim \lambda}{d!})^{2} \prod_{i=1}^n f_{\theta^{(i)}}(\lambda)
			\end{equation}		
			where $\text{dim}\lambda$ is the dimension of the irreducible representations of the symmetric group $S(d)$ corresponding to $\lambda$, and 
			\begin{align}
				f_{\theta^{(i)}}(\lambda):=\frac{d!}{z_{\theta^{(i)}}}\frac{\chi_\lambda(\theta^{(i)})}{\text{dim}\lambda}.\label{f}
			\end{align}
			is the central character of $\mu^{(i)}$. Here $\chi_\lambda(\mu^{(i)})$ is the value of the irreducible character $\chi_\lambda$ on the conjugacy class $\mu^{(i)}$. According to \eqref{Hurwitz1}, 
			\begin{align}
				H^{*}_{d}(\mu^{(1)},\dots,\mu^{(s)},\underbrace{2\,1^{d-2},2\,1^{d-2},\dots}_{k})=\sum_{\lambda\vdash d}\big(\frac{d!}{\dim\lambda}\big)^{s-2}(f_{(2,\,1^{d-2})}(\lambda))^k\prod_{i=1}^s \frac{\chi_\lambda(\mu^{(i)})}{z_{\mu^{(i)}}} .\label{5}
			\end{align}				
			A classical result \cite{Fro} of the character ratio on the transposition is
			\begin{align}
				\frac{\chi_\lambda(2,\,1^{d-2})}{\chi_\lambda(1^d)}=\frac{1}{\binom{d}{2}}\Big(\sum_i\binom{\lambda_i}{2}-\sum_i\binom{\lambda'_i}{2}\Big).\label{tran}
			\end{align}
			By \eqref{f}, \eqref{tran} and $\dim\lambda=\chi_\lambda(1^d)$, it is clear that
			\begin{align}
				f_{(2,\,1^{d-2})}(\lambda)=\sum_i\binom{\lambda_i}{2}-\sum_i\binom{\lambda'_i}{2}.\label{f2}
			\end{align}
			By \eqref{f2}, 
			\begin{align}
				|f_{(2,\,1^{d-2})}(\lambda)|<\frac{d(d-3)}{2}
			\end{align}
			unless
			\begin{align}
				&f_{(2,\,1^{d-2})}(d)=\binom{d}{2},\qquad\qquad\qquad\,\,\,\, f_{(2,\,1^{d-2})}(1^d)=-\binom{d}{2},\label{f2d}\\
				&f_{(2,\,1^{d-2})}(d-1,1)=\frac{d(d-3)}{2},\,\qquad f_{(2,\,1^{d-2})}(2,1^{d-2})=-\frac{d(d-3)}{2}.\label{f2dminus1}
			\end{align}
			According to Murnaghan-Nakayama rule \cite{M,N}, we have 
			\begin{align}
				&\chi_{(d)}(\mu)=1,\qquad\qquad\qquad\qquad\qquad\qquad\quad\,\, \chi_{(1^d)}(\mu)=(-1)^{l^*(\mu)},\label{d}\\
				&\chi_{(d-1,1)}(\mu)=m_1(\mu)-1,\qquad \chi_{(2,\,1^{d-2})}(\mu)=(m_1(\mu)-1)\cdot (-1)^{l^*(\mu)}.\label{dminus1}
			\end{align}
			\begin{lemma}\label{dH}
				For any fixed $d\geq5,s\geq0$ and $\mu^{(1)},\dots,\mu^{(s)}\vdash d$, when $k+\sum_{i=1}^s l^*(\mu^{(i)})=\text{even}$,
				\begin{align}
					H^{*}_{d}(\mu^{(1)},\dots,\mu^{(s)},\underbrace{2\,1^{d-2},2\,1^{d-2},\dots}_{k})= \frac{2}{d!^2}\prod_{i=1}^s\frac{d!}{z_{\mu^{(i)}}} &\sum_{1\leq m\leq \tbinom{d}{2}}b^*(\mu^{(1)},\dots,\mu^{(s)},m) m^{k}, \label{dHv2}
				\end{align}
				where $b^*(\mu^{(1)},\dots,\mu^{(s)},m)$ are rational numbers with
				
				1. $b^*(\mu^{(1)},\dots,\mu^{(s)},\tbinom{d}{2})=1$; 
				
				2. $b^*(\mu^{(1)},\dots,\mu^{(s)},m)=0$, for $\frac{d(d-3)}{2}<m<\tbinom{d}{2}$;
				
				3. $b^*(\mu^{(1)},\dots,\mu^{(s)},\frac{d(d-3)}{2})=(d-1)^{2-s}\prod_{i=1}^{s}(m_1(\mu^{(i)})-1)$.
			\end{lemma}		
			\begin{proof}
				By \eqref{5}, \eqref{f2} and \eqref{f2d}, we have the structure of disconnected Hurwitz numbers as follows: 
				\begin{align}
					H^{*}_{d}(\mu^{(1)},\dots,\mu^{(s)},\underbrace{2\,1^{d-2},2\,1^{d-2},\dots}_{k})=\frac{2}{d!^2}\prod_{i=1}^s\frac{d!}{z_{\mu^{(i)}}} &\sum_{1\leq m\leq \tbinom{d}{2}}b^*(\mu^{(1)},\dots,\mu^{(s)},m) m^{k} ,\nonumber
				\end{align}
				where
				\begin{align}
					b^*(\mu^{(1)},\dots,\mu^{(s)},m)=\frac{1}{2}\sum_{\substack{\lambda\vdash d\\|f_{(2,\,1^{d-2})}(\lambda)|=m}}\dim(\lambda)^2\big(\sgn(f_{(2,\,1^{d-2})}(\lambda))\big)^k\prod_{i=1}^{s}\frac{\chi_\lambda(\mu^{(i)})}{\dim(\lambda)}.      \label{top1}
				\end{align}
				Since $k+\sum_{i=1}^s l^*(\mu^{(i)})=\text{even}$, by \eqref{f2d} and \eqref{d}, $b^*(\mu^{(1)},\dots,\mu^{(s)},\tbinom{d}{2})=1$. From \eqref{f2dminus1} and \eqref{dminus1}, the second large term of $m$ in \eqref{dHv2} is $\frac{d(d-3)}{2}$ whose coefficients $b^*(\mu^{(1)},\dots,\mu^{(s)},\frac{d(d-3)}{2})$ equals to $(d-1)^{2-s}\prod_{i=1}^{s}(m_1(\mu^{(i)})-1)$.				
			\end{proof}
			By the relationship between connected and disconnected Hurwitz numbers, we have
			\begin{align}
				\sum_{g,d}&\sum_{\mu^{(1)},\dots,\mu^{(s)}\vdash d}\frac{1}{q!} H_{g,d}(\mu^{(1)},\dots,\mu^{(s)},2\,1^{d-2},2\,1^{d-2},\dots)x^{q} \prod_{i=1}^s p^i_{\mu^{(i)}}\nonumber\\
				=&\log(\sum_{k,d}\sum_{\mu^{(1)},\dots,\mu^{(s)}\vdash d}\frac{1}{k!} H^{*}_{d}(\mu^{(1)},\dots,\mu^{(s)},\underbrace{2\,1^{d-2},2\,1^{d-2},\dots}_{k})x^k \prod_{i=1}^s   p^i_{\mu^{(i)}})\label{rlt2}
			\end{align}
			with $q=2g+2d-\sum_{i=1}^s l^*(\mu^{(i)})-2$. Expanding the right hand side of \eqref{rlt2} by Taylor series and taking the coefficients of $\prod_{i=1}^s p^i_{\mu^{(i)}} x^q$ on both side,  we get
			\begin{align}
				H_{g,d}(\mu^{(1)},&\dots,\mu^{(s)},2\,1^{d-2},2\,1^{d-2},\dots)=H^*_{d}(\mu^{(1)},\dots,\mu^{(s)},\underbrace{2\,1^{d-2},2\,1^{d-2},\dots}_{q})\nonumber\\
				&-\frac{1}{2}\sum_{\substack{d_1,d_2\geq1\\d_1+d_2=d\\k_1,k_2\geq0\\k_1+k_2=q}}\sum_{\substack{\nu^{(1)},\dots,\nu^{(s)}\vdash d_1\\\sigma^{(1)},\dots,\sigma^{(s)}\vdash d_2\\\nu^{(i)}\cup\sigma^{(i)}=\mu^{(i)}\\i\in\{1,2,\dots,s\}}}\frac{q!}{k_1!k_2!}H^*_{d_1}(\nu^{(1)},\dots,\nu^{(s)},\underbrace{2\,1^{d_1-2},2\,1^{d_1-2},\dots}_{k_1})\nonumber\\
				&\qquad\qquad\qquad\qquad\qquad \times H^*_{d_2}(\sigma^{(1)},\dots,\sigma^{(s)},\underbrace{2\,1^{d_2-2},2\,1^{d_2-2},\dots}_{k_2})+\cdots.\nonumber
			\end{align}
			Here, the $\cdots$ term does not contain any $m^q$ with $m\geq \binom{d-1}{2}$. Notice that for $d_1,d_2\geq 1$ such that $d_1+d_2=d$, it holds that
			\begin{align}
				\binom{d_1}{2}+\binom{d_2}{2}\leq\binom{d-1}{2}.\nonumber
			\end{align}
			Using the binomial theorem, Theorem \ref{cH2} follows from Lemma \ref{dH}.
		\end{proof}

		\section*{Acknowledgement}
		I would like to thank Di Yang for his advice. This work was supported by NSFC No. 12371254 and CAS No. YSBR-032.

	\end{CJK}
\end{document}